\documentclass[a4paper,12pt]{amsart}
\usepackage{mathtools}
\usepackage{amsmath, amssymb, amscd, amsthm, amsfonts, amsbsy}
\usepackage{array}
\usepackage{epsfig}
\usepackage{ragged2e}
\usepackage{graphics}
\usepackage{fancyhdr}
\usepackage[all,cmtip]{xy}
\usepackage{ifthen}
\usepackage[pdftex,bookmarks=true]{hyperref}
\nonstopmode \numberwithin{equation}{section}
\makeatletter
\newcommand{\ubrace}[2]{\mathord{\mathpalette\ubrace@{{#1}{#2}}}}
\newcommand{\ubrace@}[2]{\ubrace@@#1#2}
\newcommand{\ubrace@@}[3]{
	\underbrace{#1#2}_{#3}%
}
\makeatother

\theoremstyle{definition}

\newtheorem{Thm}{Theorem}[section]
\newtheorem{Cor}[Thm]{Corollary}
\newtheorem{Lem}[Thm]{Lemma}

\theoremstyle{definition}
\newtheorem{Def}[Thm]{Definition}

\newtheorem{Rem}[Thm]{Remark}

\begin{document}
	\bibliographystyle{amsplain}
	\title{Center and torsion of a quotient of the group of piecewise linear homeomorphisms of the real line}
	\author{Swarup Bhowmik}
	\address{Swarup Bhowmik, Department of Mathematics,
		Indian Institute of Technology Kharagpur, Kharagpur - 721302, India.}
	\email{swarup.bhowmik@iitkgp.ac.in}
	\begin{abstract}
		In this article, we show that the group comprising piecewise-linear homeomorphisms of the real line with bounded slopes is not simple. Furthermore, we establish that a quotient of this group is torsion-free, and importantly, the center of that quotient group is indeed trivial.
	\end{abstract}
	\thanks{The author of this article acknowledges the financial support from Inspire, DST, Govt. of India as a Ph.D. student (Inspire) sanction
		letter number: No. DST/INSPIRE Fellowship/2018/IF180972}
	\maketitle
	\pagestyle{myheadings}
	\markboth{S. Bhowmik}{Center and torsion of a quotient of the group of piecewise linear homeomorphisms of the real line}
	\bigskip
	\section{Introduction}
	The concept of quasi-isometry is one of the fundamental concepts in geometric group theory.  For a given metric space $X$, the group of quasi-isometries from $X$ to itself is denoted as $QI(X)$ and serves as a quasi-isometric invariant of $X$. Our primary interest lies in cases where $X$ is quasi-isometrically equivalent to a finitely generated group $\Gamma$ equipped with a word metric.  In the context of finitely generated groups, selecting a finite generating set $S\subset \Gamma$ introduces the word metric, denoted as $dS$, which endows $\Gamma$ with a metric space structure. In the field of studying the coarse geometry of $\Gamma$, $QI(\Gamma)$ plays a vital role as a significant invariant. However, the determination of $QI(\Gamma)$ for arbitrary groups remains a challenging task, and very limited knowledge is available concerning these groups in general. 
	
	It appears that for a very few families of groups $\Gamma$, $QI(\Gamma)$ has been explicitly investigated, for example, solvable Baumslag-Solitar groups $BS(1,n)$ \cite[Theorem 7.1]{Farb Mosher}, the groups $BS(m,n),~1<m<n$ \cite[Theorem 4.3]{Whyte}, irreducible lattices in semisimple Lie groups (see \cite{Farb} and the references therein) etc. Even the exploration of $QI(\mathbb{R})$, the group of quasi-isometries of the real line remains relatively limited. Nevertheless, Gromov and Pansu demonstrated in \cite[\textsection 3.3.B]{Gromov Pansu} that the map Bilip$(\mathbb{R})\rightarrow QI(\mathbb{R})$ is surjective and established that $QI(\mathbb{R})$ is an infinite-dimensional group.
	
	In the context of coarse geometry, the group of quasi-isometries of the real line, holds a crucial importance. It is known that several remarkable geometric groups can be embedded in $QI(\mathbb{R})$ including:- (i) $\widetilde{Diff(\mathbb{S}^{1})}$ and $\widetilde{PL(\mathbb{S}^{1})}$, (ii) $PL_k(\mathbb{R})$, (iii) the Thompson's group $F$, (iv) the free group of rank $c$ (the continuum). Notably, the group $PL_k(\mathbb{R})$ is simple (see \cite{Brin Squier} and Theorem 3.1 of \cite{Epstein}). The group $\widetilde{Diff(\mathbb{S}^{1})}$ contains a free group of rank equal to the continuum \cite{Grabowski}. The Thompson's group $F$ exhibits numerous remarkable properties, such as its commutator subgroup $[F,F]$ being a simple group, and $F$ not containing a non-abelian free subgroup \cite{Cannon Floyd Parry}. Hence, $QI(\mathbb{R})$ contains a diverse set of subgroups with exceptional properties.
	
	An explicit characterization of all quasi-isometries of $\mathbb{R}$ was provided in \cite{Sankaran}, which established the existence of a surjective homomorphism from $PL_\delta(\mathbb{R})$ to $QI(\mathbb{R})$, where $PL_\delta(\mathbb{R})$ be the set of all piecewise-linear homeomorphisms $f$ of $\mathbb{R}$ with bounded slopes. In \cite{Sankaran}, Sankaran provides a concise and well-understood representation of the elements of $QI(\mathbb{R})$. But there is currently no complete explicit representation of $QI(\mathbb{R})$ though some attempts have been taken in \cite{Bhowmik Chakraborty 2}. This is an almost complete characterization of the group $QI(\mathbb{R})$.
	
	In this research article, our main focus is on the set $PL_\delta(\mathbb{R}_{+})/H$ where $H=\Big\{f\in PL_\delta(\mathbb{R}_{+}): \displaystyle\lim_{x\rightarrow\infty}\displaystyle\frac {f(x)}{x}=1\Big\}$ analogous to the defintion of $H$ in \cite{Ye Zhao}, which holds significant importance in characterizing the group $QI(\mathbb{R}_{+})/H$. We established three essential algebraic and geometric properties of the group $PL_\delta(\mathbb{R}_{+})/H$, which are crucial for the broader study of the group $QI(\mathbb{R}_{+})/H$. Recognizing the inherent challenges in characterizing the geometric properties of $QI(\mathbb{R})$, our study on $PL_\delta(\mathbb{R}_{+})/H$ offers a more reliable approach to explore these fundamental algebraic and geometric properties. To be more precise, we present the following results. The notations used in this context will be explained in more detail in \textsection\ref{Preliminaries}.
	\begin{Thm}\label{Non-simple}
		\textit{The group $PL_\delta^{+}(\mathbb{R})$ (or $PL_\delta([0,+\infty])$) is not simple.}
	\end{Thm}
    \begin{Thm}\label{Torsion-free}
    	\textit{The quotient group $PL_\delta(\mathbb{R}_{+})/H$ is torsion-free.}
    \end{Thm}
	\begin{Thm}\label{Center}
		\textit{The center of the quotient group $PL_\delta(\mathbb{R}_{+})/H$ is trivial.}
	\end{Thm}
	
	In this paper, we will consistently denote the Euclidean metric as $d$ throughout our discussions. Additionally, for the sake of convenience in notation, we will use $G$ to represent $PL_\delta(\mathbb{R}_{+})$.

	\section{Preliminaries}\label{Preliminaries}
	\subsection{$PL_\delta$-homeomorphisms of $\mathbb{R}$} Let $f:\mathbb{R}\rightarrow\mathbb{R}$ be a homeomorphism of the real numbers. Let us denote $B(f)$ as the set of breakpoints of $f$, which are points where $f$ lacks a derivative, and let $\Lambda(f)$ denote the set of slopes associated with $f$, defined as $\Lambda(f)=\{f'(t):t\in \mathbb{R}\setminus B(f)\}$. It is important to note that if $f$ is piecewise differentiable, then $B(f)$ is a discrete subset of the real numbers $\mathbb{R}$.
	\begin{Def}
	We define a subset $\Lambda$ of $\mathbb{R}^{*}$ (the set of non-zero real numbers) as \textit{bounded} if there exists an $K>1$ such that $K^{-1}<|\lambda|<K$ for all $\lambda\in\Lambda$. Furthermore, for a homeomorphism $f$ of $\mathbb{R}$ that is piecewise differentiable, we say that $f$ has bounded slopes if the set $\Lambda(f)$ is bounded.
\end{Def}
    Let $PL_\delta(\mathbb{R})$ be the set comprising all piecewise-linear homeomorphisms $f$ of $\mathbb{R}$ such that $\Lambda(f)$ is bounded. Notably, $PL_\delta(\mathbb{R})$ forms a subgroup within the larger group $PL(\mathbb{R})$, which encompasses all piecewise-linear homeomorphisms of $\mathbb{R}$.

   Recall that a map $f:(X,d)\rightarrow (X,d)$ is \textit{bi-Lipschitz} if there exists
   a constant $k\geq 1$ such that $\displaystyle\frac {1}{k}d(x, y)\leq d(f(x),f(y))\leq kd(x,y).$ The set of all such bi-Lipschitz homeomorphisms from $X$ to itself is denoted as Bilip$(X)$. Importantly, it is easy to see that if a function $f\in PL_\delta(X)$, then $f\in$ Bilip($X$). (See Lemma \ref{Bilip} below.)
   
\subsection{Quasi-isometry} Consider a metric space $(X, d).$ A quasi-isometry of $X$ is a map $f:X\rightarrow X$ for which there exist constants $\mu\geq 1$, $\delta>0$, and $M>0$ such that the following criteria are met:\\ (i) For any pair of points $x,y\in X$, $-\delta+\Big(\displaystyle\frac {1}{\mu}\Big)d(x,y)\leq d(f(x),f(y))\leq \mu d(x,y)+\delta$, (ii) every point $x\in X$ is at a distance not greater than $M$ from the entire set $f(X)$.

It is note that when $\delta=0$, the quasi-isometry $f$ is classified as a bi-Lipschitz mapping. The collection of all quasi-isometries of $X$ forms a set that is closed under composition, yet it does not constitute a group because quasi-isometries are not necessarily one-to-one or onto, and therefore, they may lack inverses. As an illustration, the floor function, which maps $x\mapsto\lfloor x \rfloor$, is a quasi-isometry of the real numbers $\mathbb{R}$ but does not have an inverse.

Two mappings $f$ and $g$ : $X\rightarrow X$ are considered quasi-isometrically equivalent if there exists a constant $C>0$ such that $d(f(x), g(x))<C$, for all $x\in X$. We denote the equivalence class of a quasi-isometry $f:X\rightarrow X$ as $[f]$, although sometimes $f$ is used instead of $[f]$ for notation convenience. The collection of all equivalence classes of quasi-isometries of $X$, denoted as $QI(X)$, forms a group under composition, where $[f].[g]=[f\circ g]$, for $[f]$ and $[g]\in QI(X).$

Moreover, any quasi-isometry $f:X\rightarrow X$ induces an isomorphism from $QI(X)$ to $QI(X)$, defined as $[g]\rightarrow [f\circ g\circ h]$, where $h:X\rightarrow X$ is a quasi-inverse of $f$. For instance, the map $x\rightarrow [x]$ is a quasi-isometry from $\mathbb{R}$ to $\mathbb{Z}$. As a result, $QI(\mathbb{R})$ is isomorphic to $QI(\mathbb{Z})$ as a group. The quasi-isometries of the real line can be effectively represented by PL-homeomorphisms of the real line characterized by bounded slopes \cite{Sankaran}.

\subsection{Notations} We denote $Diff(\mathbb{S}^{1})$ as the group of diffeomorphisms and $PL(\mathbb{S}^{1})$ as the group of piecewise linear homeomorphisms of $\mathbb{S}^{1}$. Importantly, for any homeomorphism $f$ of $\mathbb{S}^{1}$ can be lifted to a homeomorphism $\tilde{f}$ of $\mathbb{R}$, ensuring the commutativity of the diagram below:
\begin{center}
	\includegraphics[scale=1.2]{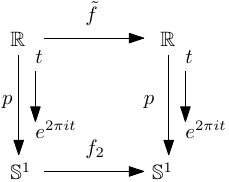}
\end{center}
We denote $\widetilde{Diff(\mathbb{S}^{1})}$ and $PL(\mathbb{S}^{1})$ as the groups of all lifts of elements from $Diff(\mathbb{S}^{1})$ and $PL(\mathbb{S}^{1})$ respectively. Additionally, we represent the group of piecewise linear homeomorphisms of $\mathbb{R}$ with compact support as $PL_k(\mathbb{R})$.


	\section{Structure of $PL_\delta(\mathbb{R})$}\label{Sturcture}
	Let $PL^{+}_{\delta}(\mathbb{R})$ be the subgroup of $PL_\delta(\mathbb{R})$ consisting of orientation preserving PL-homeomorphisms. Then we have $PL_\delta(\mathbb{R})$ = $PL^{+}_{\delta}(\mathbb{R})\rtimes <t>$
	where $t$ is the reflection of $\mathbb{R}$ about the origin. Let $PL_\delta(\mathbb{R}_{+})$ (resp. $PL_\delta(\mathbb{R}_{-})$) be the piecewise-linear homeomorphism with bounded slopes of the half line $[0,\infty)$ (resp. $(-\infty,0]$). Note that $PL^{+}_{\delta}(\mathbb{R})=PL_\delta(\mathbb{R}_{+})\times PL_\delta(\mathbb{R}_{-})$. In order to characterize the properties of the group $PL^{+}_{\delta}(\mathbb{R})$, we introduce a new invariant $S_f$ for any $f\in PL_\delta(\mathbb{R})$ as follows:\\ $S_{f}=\Big\{\displaystyle\lim_{n\rightarrow\infty}\frac {f(x_n)}{x_n}:x_n\rightarrow\infty~\text{and}~ \frac {f(x_n)}{x_n}~\text{converges}\Big\}$ analogous to the definition of $S_{[g]}$ in \cite{Bhowmik Chakraborty 2}. Now, we established a crucial relationship between the group $PL_\delta(\mathbb{R})$ and Bilip$(\mathbb{R})$ which unveils the structure of $PL_\delta(\mathbb{R})$.

	\begin{Lem}\label{Bilip}
		\textit{Let $f$ be a piecewise differentiable homeomorphism of $\mathbb{R}$ with $\Lambda(f)$ is bounded. Then $f$ is bi-Lipschitz.}
	\end{Lem}
\begin{proof}
	Suppose that $\displaystyle\frac {1}{K}<\Lambda(f)<K$. If $f$ is differentiable everywhere, then clearly by defintion of $\Lambda(f)$, we have
	\begin{equation}\label{Bilip 1}
		\displaystyle\frac {1}{K}d(x,y)\leq d(f(x),f(y))\leq Kd(x,y).
	\end{equation} 
	Suppose that $B(f)\neq \phi$. Let $x_1,x_2,\cdots,x_k$ be the points where $f$ is non-differentiable such that $x_1<x_2<\cdots<x_k$. Thus within each interval $(x_i,x_{i+1})$, for $i=1,2,\cdots,k-1$, $f$ is differentiable. Then by applying the Mean Value Theorem in each of these intervals there exist $c_i\in(x_i,x_{i+1})$ such that $|f(x_{i+1}-f(x_i))|=|f'(c_i)||x_{i+1}-x_i|$. Since $\Lambda(f)\subset\Big(\displaystyle\frac {1}{K},K\Big)$, therefore, we have
	\begin{equation}\label{Bilip 2}
		\displaystyle\frac {1}{K}d(x_{i},x_{i+1})\leq d(f(x_i),f(x_{i+1}))\leq Kd(x_i,x_{i+1}),~  i=1,2,\cdots,k-1.
	\end{equation} 
	Now for any $y>x$, where $x<a_1<a_2<\cdots<a_k<y$, applying the Mean Value Theorem once more, we can verify that $f(y)-f(x)=\displaystyle\sum_{i=0}^{k}f'(z_i)(a_{k+1}-a_k)$, for some $z_i\in(a_{i},a_{i+1})$ where $a_0=x$ and $a_{k+1}=y.$ As $\displaystyle\frac {1}{K}<\Lambda(f)<K$, it follows that
	\begin{equation}\label{Bilip 3}
		\displaystyle\frac {1}{K}d(x,y)\leq d(f(x),f(y))\leq Kd(x,y)
	\end{equation} 
Therefore, by combining (\ref{Bilip 1}),(\ref{Bilip 2}) and (\ref{Bilip 3}), we can conclude that the function $f$ is bi-Lipschitz. 
\end{proof}
\begin{Cor}
	It can be easiliy verified from the above Lemma and Lemma 3.1 of \cite{Bhowmik Chakraborty 2} that for any $f\in PL_\delta(\mathbb{R})$, $S_f$ is either a singleton set or a closed and bounded interval.
\end{Cor}

\section{Algebraic and Geometric properties}
In this section, we begin by demonstrating that the group $PL_\delta(\mathbb{R}_{+})$ is not simple by providing a normal subgroup $H$ defined in the Introduction. Subsequently, utilizing the invariant $S_f$, we conclude that the group $PL_\delta(\mathbb{R}_{+})/H$ is torsion-free. A fundamental preliminary aspect for exploring the algebraic properties of $PL_\delta(\mathbb{R}_{+})/H$ is to identify the center of $PL_\delta(\mathbb{R}_{+})/H$. Consequently, we establish that the center of the group $PL_\delta(\mathbb{R}_{+})/H$ is indeed trivial.

\begin{Lem}
	\textit{The set $H$ is a proper normal subgroup of $G$.}
\end{Lem}
\begin{proof}
	First we show that $H$ is a subgroup of $G$ and then we show that it is normal in $G$. For any  $f,g\in H$, we get
	\begin{equation*}
		\frac {f(g(x))}{x}-1=\frac {f(g(x))-g(x)}{g(x)}\frac {g(x)}{x}+\frac {g(x)-x}{x}
	\end{equation*}
	As $f,g\in H$ and since $x\rightarrow\infty$ implies $g(x)\rightarrow\infty$, therefore, when $x\rightarrow\infty$, $\displaystyle\frac {f(g(x))-g(x)}{g(x)}$, $\displaystyle\frac {g(x)-x}{x}\rightarrow 0$ and $\displaystyle\frac{g(x)}{x}$$\rightarrow$ 1. Then from above we get $fg\in H$.\\
	Now, without loss of generality, we can take that $f(x)\neq x$. Then we have, 
	\begin{equation*}
		\frac {f^{-1}(x)}{x}-1=\frac {f^{-1}(x)-f^{-1}(f(x))}{x-f(x)}\frac {x-f(x)}{x}
	\end{equation*}
	Since $f^{-1}\in G$, therefore $\displaystyle\frac {f^{-1}(x)-f^{-1}(f(x))}{x-f(x)}$ is bounded and as $f\in H$, so as $x\rightarrow\infty$, $\displaystyle\frac {f(x)-x}{x}\rightarrow 0$. Thus $f^{-1}\in H$ and that $H$ is a subgroup of $G$.\\
	To prove $H$ is normal in $G$, let $g\in G$ and $f\in H$. Then we have,
	\begin{equation}\label{Normal condition}
		\frac {g^{-1}(f(g(x)))}{x}-1=\frac {g^{-1}(f(g(x)))-g^{-1}(g(x))}{g(x)}\frac {g(x)}{x}
	\end{equation}
	Since $g\in G$, therefore as $x\rightarrow\infty$, the function $\displaystyle\frac {g(x)}{x}$ is bounded. Then by putting $g(x)=z$ and assuming $f(z)\neq z$, we get from above,
	\begin{equation*}
		\frac {g^{-1}(f(g(x)))-g^{-1}(g(x))}{g(x)}=\frac {g^{-1}(f(z))-g^{-1}(z)}{f(z)-z}\frac {f(z)-z}{z}
	\end{equation*} 
	Now, as $g\in G$, so $g^{-1}\in G$. Therefore, $\displaystyle\frac {g^{-1}(f(z))-g^{-1}(z)}{f(z)-z}$ is bounded. Again as $x\rightarrow\infty$ implies $z\rightarrow\infty$ and so $\displaystyle\frac {f(z)-z}{z}\rightarrow 0$. Then from (\ref{Normal condition}) we get, \\
	$\displaystyle\frac {g^{-1}(f(g(x)))}{x}\rightarrow 1$ as $x\rightarrow\infty$. Thus $g^{-1}fg\in H$.\\
	Now, it is clear that the functions $f(x)=kx,~k>1$ is not an element of $H$. Hence we can conclude that the set $H$ is a proper normal subgroup of $G$. 
\end{proof} 
\begin{Rem}
	One has a natural homomorphism $\psi:PL_\delta(\mathbb{R}_{+})\rightarrow QI(\mathbb{R}_{+})$, where $\psi(f)=[f]$, for all $f\in PL_\delta(\mathbb{R}_{+})$ and this homomorphism is surjective has been proven by Sankaran in \cite{Sankaran}. Furthermore, there always exists a natural epimorphism from $QI(\mathbb{R}_{+})\rightarrow QI(\mathbb{R}_{+})/H$. Therefore, we can assert that there is an onto homomorphism from $PL_\delta(\mathbb{R}_{+})\rightarrow QI(\mathbb{R}_{+})/H$. So, this leads us to a good question whether there exists a surjective homomorphism from $PL_\delta(\mathbb{R}_{+})/H\rightarrow QI(\mathbb{R}_{+})/H$.
\end{Rem}
Now, we are ready to prove Theorem \ref{Non-simple}.

\subsection{Proof of Theorem \ref{Non-simple}} Since we know that the two groups $PL_\delta(\mathbb{R}_{+})$ and $PL_\delta(\mathbb{R}_{-})$ are isomorphic and  $PL^{+}_{\delta}(\mathbb{R})=PL_\delta(\mathbb{R}_{+})\times PL_\delta(\mathbb{R}_{-})$, therefore the conclusion of Theorem \ref{Non-simple} immediately follows from the preceding lemma.\hspace{2.8cm}\qedsymbol{}\\

In the field of group theory torsion plays a pivotal role, particularly in the study of geometric groups and their structure. It helps to analyze and characterize different types of groups by examining the torsion subgroup and the order of elements. Now, we show that the group $PL_\delta(\mathbb{R}_{+})/H$ is torsion-free.

\subsection{Proof of Theorem \ref{Torsion-free}} Let us consider an arbitrary element $gH$ within the quotient group $G/H$, where $g\in G\setminus H$. To establish the torsion-free nature of the quotient group $G/H$, it suffices to show that the order of the element $gH$ is infinite, signifying that no positive integer $r$ exists for which $g^r$ belongs to $H$. \\
As $g$ belongs to $G$ but not to $H$, we can assert the existence of a sequence $\{x_n\}$ for which $\displaystyle\frac {g(x_n)}{x_n}$ does not converge to 1. To establish that $g^r$ is not an element of $H$, it is adequate to find a sequence $\{y_n\}$ such that $\displaystyle\frac {g^r(y_n)}{yn}$ does not converge to 1, for any positive integer $r$.\\ 
Now, let's consider the set $S_{g}$, which may take the form of $[m,1],[1,M]$ or $[m,M]$. \\
If $S_{g}=[m,1]$, then there exists a sequence $\{z_n\}$ such that $\displaystyle\frac {g(z_n)}{z_n}\rightarrow m<1$.\\
Now, $\displaystyle\frac {g^{r}(z_n)}{z_n}=\displaystyle\frac {g(g^{r-1}(z_n))}{g^{r-1}(z_n)}\displaystyle\frac {g(g^{r-2}(z_n))}{g^{r-2}(z_n)}\cdots \frac {g(z_n)}{z_n}$. Then there exists a subsequence $\{y_n'\}$ of $\{z_n\}$ such that $\displaystyle\frac {g^{r}(y_n')}{y_n'}\nrightarrow 1$ as $n\rightarrow\infty$. Similarly, if $S_{g}$ is represented as $[1,M]$, there also exists a sequence $\{y_n^{''}\}$ such that $\displaystyle\frac {g^{r}(y_n^{''})}{y_n^{''}}\nrightarrow 1$ as $n\rightarrow\infty$.\\
In the last case where $S_{g}=[m,M]$, with $m<1<M$, let us assume that there is an element $m'$ such that $m<m'<1<M$. Then there exists a sequence $\{x_n^{'}\}$ such that $\displaystyle\frac {g(x_n^{'})}{x_n^{'}}\rightarrow m<m'<1$ as $n\rightarrow\infty$. Consequently there exist $N_1,N_2\in \mathbb{N}$ such that $g(x_n^{'})<x_n^{'}$ for all $n\geq N_1$ and $g(x_n^{'})<m'~x_n^{'}$ for all $n\geq N_2$. Let $N=$max$\{N_1,N_2\}$.

Since $g$ is orientation preserving homeomorphism, for each $n\geq N$, we have
$g^r(x_n^{'})<g^{r-1}(x_n^{'})<\dots<g(x_n^{'})<m'~x_n^{'}$. Then we can choose a subsequence $\{y_n'''\}$ of $\{x_n^{'}\}$ such that $\displaystyle\frac {g^{r}(y_n''')}{y_n'''}\nrightarrow 1$ as $n\rightarrow\infty$.\\
Consequently, in any cases and for any positive integer $r$, we can always construct a sequence $\{y_n\}$ for which $\displaystyle\frac {g^{r}(y_n)}{y_n}\nrightarrow 1$ as $n\rightarrow\infty$. This implies that $g^{r}\notin H$ for any $r$, leading to the conclusion that the group $G/H$ is indeed torsion-free.\hspace{1.8cm}\qedsymbol{}\\

Now, initially, we establish that the center of the quasi-isometry group $QI(\mathbb{R}_{+})/H$ is trivial. Subsequently, by utilizing this result, we further confirm that the center of the group $PL_\delta(\mathbb{R}_{+})/H$ is also trivial. 
\begin{Lem}
	\textit{The center of the quasi-isometry group $QI(\mathbb{R}_{+})/H$ is trivial.} 
\end{Lem}
\begin{proof}
	Let $[f]$ represent a non-identity element within the quotient group $QI(\mathbb{R}_{+})/H$. Then by using Lemma 3.1 of \cite{Chakraborty}, we are able to assume the existence of a sequence $\{a_n\}$ such that $\{a_n\}\rightarrow+\infty$, $a_{n+1}>3f(a_n)$, and $f(a_n)-a_n\rightarrow+\infty$.\\
	We proceed to establish the existence of an element $[g]$ in $QI(\mathbb{R}_{+})/H$ such that $[f]$ and $[g]$ do not commute. Specifically, we have a piecewise-linear homeomorphism $g$ as defined by \cite{Chakraborty}, satisfying the following conditions:\\
	\vspace{2mm}
	$(a)$ $g(x)=x,~x\leq a_1$,\\
	\vspace{2mm}	
	$(b)~g(I_k)=I_k,~I_k:=[a_k,a_{k+1}]$,\\
	\vspace{2mm}	
	$(c)$ $g$ has exactly one break point at $f(a_k)\in I_k$ and $g(f(a_k))=\displaystyle\frac {a_k+f(a_K)}{2}$.\\
	\vspace{2mm}
	It is evident from the definition that $g'(t)\in\Big[\displaystyle\frac {1}{2},\displaystyle\frac {5}{4}\Big]$, showing that $g\in PL_\delta(\mathbb{R})$. Furthermore, we observe that $\displaystyle\lim_{k\rightarrow\infty}\frac {g(f(a_k))}{f(a_k)}=\lim_{k\rightarrow\infty}\frac {f(a_k)+a_k}{2f(a_k)}=\frac {1}{2}+\frac {1}{2}\frac {1}{\displaystyle\lim_{k\rightarrow\infty}\frac {f(a_k)}{a_k}}\nrightarrow 1.$ \\
	Thus, $[g]\in QI(\mathbb{R}_{+})/H$. Now,
	\begin{equation*}
	\Big|\frac {(f\circ g)(a_k)}{a_k}-\frac {(g\circ f)(a_k)}{a_k}\Big|=\Big|\frac {f(a_k)}{a_k}-\frac {a_k+f(a_k)}{2a_k}\Big|=\Big|\frac {f(a_k)}{2a_k}-\frac {1}{2}\Big|\nrightarrow 0. 
	\end{equation*}
Hence, in accordance with Lemma 3.4 of \cite{Bhowmik Chakraborty 2}, $[f\circ g]\neq [g\circ f]$ (mod $H$).\\
This completes the proof.		 
\end{proof}
Now, we are ready to prove \ref{Center}.
\subsection{Proof of Theorem \ref{Center}} Similar to the verification in Lemma 3.4 of \cite{Bhowmik Chakraborty 2}, it can be readily checked that for any $f,g$ in $PL_\delta(\mathbb{R}_{+})$, they belong to the same equivalece class of $H$ if and only if there exists an $M>0$ such that $\Big|\displaystyle\frac {f(x)}{x}-\displaystyle\frac {g(x)}{x}\Big|<\epsilon$, for all $x>M$ and for a given $\epsilon>0$. The conclusion of the Theorem \ref{Center} immediately follows by applying this condition and from the previous lemma.
\section*{Acknowledgement}
The author thanks Prateep Chakraborty for his valuable suggestions and comments.

\end{document}